\documentclass[12pt]{amsart}
\usepackage{amsfonts}
\usepackage{graphicx,amssymb,amsfonts,amsmath,amscd}
\usepackage[all,ps,cmtip]{xy}
\usepackage{amscd}

\usepackage{mathrsfs}

\usepackage{eucal}
\usepackage{hyperref}

\newtheorem{defn}{Definition}[section]
\newtheorem{thm}[defn]{Theorem}
\newtheorem{lem}[defn]{Lemma}

\newtheorem{cor}[defn]{Corollary}
\newtheorem{prop}[defn]{Proposition}
\newtheorem{rem}[defn]{Remark}

\newcommand{\N}{\mathbb{N}}
\newcommand{\R}{\mathbb{R}}
\newcommand{\C}{\mathbb{C}}
\newcommand{\Z}{\mathbb{Z}}

\newcommand{\T}{\mathbb{T}}

\newcommand{\n}{\|}

\input xy
\xyoption{all}

\begin{document}
\title{Completely bounded Paley projections on anisotropic Sobolev spaces on tori }
\author{Yanqi QIU}

\begin{abstract}
We study the existence of certain completely bounded Paley projection on the anisotropic Sobolev spaces on tori. Our result should be viewed as a generalization of a similar result obtained by Pe{\l}czy{\'n}ski and Wojciechowski in \cite{PW}. By a transference method, we obtain similar results on the Sobolev spaces on  quantum tori.
\end{abstract}
\maketitle

\section{Introduction}
Let $S \subset \N^d$ be a finite subset, containing the origin and satisfying some saturation conditions. The anisotropic Sobolev space $W^S_1(\T^d)$ is defined via the norm $$\n f \n_{S, 1} =  \sum_{\gamma \in S} \n \partial^\gamma f\n_{L_1(\T^d)}.$$
In \cite{PW}, necessary and sufficient conditions on $S$ are given under which there exist the so-called Paley projections on $W^S_1(\T^d)$. 

By the definition, $W^S_1(\T^d)$ embeds isometrically in $\ell_1^{|S|}(L_1(\T^d))$, it is well-known that on the latter space, there exists a natural operator space structure, and we will equip $W^S_1(\T^d)$ with the sub-operator space structure via the above embedding.

Following the proofs in \cite{PW}, we show that under the same conditions on $S$, the projections considered by Pe{\l}czy{\'n}ski and Wojciechowski are in fact completely bounded. The complete boundedness of these projections can be applied to obtain similar results on the Sobolev space $W^S_1(\T_\theta^d)$ associated to the quantum torus $\T^d_\theta$.

\section{Prelininaries}
Denote by $\N$ the set of non-negative integers. Fix a positive integer $d \ge 1$. The usual scalar product on the Euclidian space $\R^d$ is denoted by $\langle \cdot, \cdot \rangle$. We denote by $\T^d$ the group $(\R/2 \pi \Z)^d$ equipped with its normalized Haar measure $dx$, it will be identified with the cube $[-\pi, \pi)^d$ in a standard way. The dual group of $\T^d$ is $\Z^d$ such that to each $n \in \Z^d$ is assigned the character $\chi_n: \T^d \to \C$ defined by $\chi_n(x) = e^{i \langle x, n\rangle}$. Trigonometric polynomials are complex linear combinations of characters. The set of trigonometric polynomials on $\T^d$ is denoted by $\mathscr{P}_d$.

To each $\gamma = (\gamma(j)) \in \N^d$, we associate with the partial derivative $$\partial^\gamma= \frac{\partial^{|\gamma|}}{\partial^{\gamma(1)}_{x(1)} \partial^{\gamma(2)}_{x(2)} \dots \partial^{\gamma(d)}_{x(d)}},$$ where $| \gamma | = \gamma(1) + \gamma(2) + \cdots \gamma(d)$.
 
A smoothness $S$ is a finite subset of $\N^d$ which contains the origin 0, and such that: if $\alpha= (\alpha(j)) \in S$ then every $\beta = (\beta(j)) \in \N$ such that $\beta(j) \le \alpha(j)$ for $j = 1, 2, \cdots, d$ belongs to $S$.

For each $\gamma \in \N^d$, we define the symbol $\sigma_\gamma: \R^d \to \C$ as the function: if $x= (x(j)) \in (\R\setminus 0)^d$, then $$\sigma_\gamma(x) = i^{|\gamma|} x^\gamma = \prod_{j = 1}^d (i x(j))^{\gamma(j)},$$ otherwise, $\sigma_\gamma(x) = 0$.

The fundamental polynomial of a smoothness $S$ is $$Q_S = \sum_{\gamma \in S} | \sigma_\gamma |^2,$$ which is a non-negative function on $\R^d$.

The Sobolev space $W^S_p(\T^d)$ is defined as the completion of $\mathscr{P}_d$ with respect to the norm defined as following: if $f \in \mathscr{P}_d$, then \begin{eqnarray}\label{def_sob}\n f \n_{S, p} : = \Big(\sum_{\gamma \in S} \n \partial^\gamma f \n_{L_p(\T^d)}^p\Big)^{1/p}.\end{eqnarray}

\begin{rem}
The original definition of $\n f\n_{S,p}$ in \cite{PW} is $$\n f \n_{S,p} = \Big(\int_{\T^d} \Big(\sum_{\gamma \in S} | \partial^\gamma f(x) |^2\Big)^{p/2} dx\Big)^{1/p},$$ which is equivalent to our definition since $S$ is a finite set.
\end{rem}

Let $f \in L_1(\T^d)$, its spectrum  $\text{spec}(f)$ is $$\text{spec}(f) : = \{ n \in \Z^d: \hat{f}(n) = \int_{\T^d} f(x) e^{-i \langle x, n \rangle} dx \ne 0\}.$$  

Let $\Lambda \subset \Z^d$ be an infinite subset. The projection $P_\Lambda : \mathscr{P}_d \rightarrow \mathscr{P}_d$ is defined by  $P_\Lambda f = \sum_{n \in \Lambda} \hat{f}(n) e^{i \langle \cdot, n \rangle}$.
  
\begin{defn}
In the above situation, $P_\Lambda$ will be called a Paley projection if there is some $K > 0$, such that $$ \n P_\Lambda f \n_{S,2} \le K \n f \n_{S,1}, \text{\quad for all $f\in \mathscr{P}_d $,}$$ i.e. for all $f \in \mathscr{P}_d$, we have $$\Big(\sum_{n \in \Lambda} Q_S(n) | \hat{f}(n) |^2\Big)^{1/2} \le K \n f \n_{S,1}.$$
\end{defn}

If $P_\Lambda$ is a Paley projection, then the natural mapping $W^S_2(\T^d)_\Lambda \rightarrow W^S_1(\T^d)_\Lambda$ is an isomorphism. $P_\Lambda$ can be uniquely extended to be an projection on $W^S_1(\T^d)$, which is still denoted by $P_\Lambda: W^S_1(\T^d) \rightarrow W^S_1(\T^d).$

For the operator space theory, we refer to the book \cite{Pisier1} for a detailed study. Here we recall that the usual $L_p$-spaces are equipped with a natural operator space structure (in short o.s.s. For the detail, see e.g.\cite{Pisier1} p.178 -p.180). Hence $W^S_1(\T^d)$ is an operator space by the embedding $W^S_1(\T^d) \subset \ell_1^{|S|}(L_1(\T^d))$. 

We will use the following useful fact: Let $E \subset L_1(\Omega, \mu)$ and $F \subset L_1(M, \nu)$ be two operator subspaces. Then a linear operator $u: E \rightarrow F$ is completely bounded iff $u \otimes I_{S_1} : E(S_1) \rightarrow F(S_1)$ is bounded, where $S_1$ is the set of trace class operators and $E(S_1)$ and $F(S_1)$ are the closures of $E \otimes S_1$ and $F \otimes S_1$ in $L_1(\Omega, \mu; S_1)$ and $L_1(M, \nu; S_1)$ respectively. Moreover, $$ \n u \n_{cb} = \n u \otimes I_{S_1}\n.$$ 

Recall that the operator space $C + R$ is a homogeneous Hilbertian operator space, which is determined by the following fact: if $(e_k)$ is an orthonormal basis of $C + R$ and $(x_k)$ is a finite sequence in $S_1$, then $$\n \sum_k x_k \otimes e_k \n_{S_1[C+ R]} = \inf\{\n (\sum_k y_k^* y_k)^{1/2}\n_{S_1}+  \n (\sum_k z_k z_k^* )^{1/2}\n_{S_1} \},$$ where the infimum runs over all possible decompositions $ x_k = y_k + z_k$. (For the definition of $S_1[E]$, see \cite{Pisier2}). For convience, we will denote $$||| (x_k) ||| : = \n \sum_k x_k \otimes e_k \n_{S_1[C+ R]}.$$

The following theorem of Lust-Piquard and Pisier will be used in this note. 

\begin{thm}(Lust-Piquard {\&} Pisier)
Let $(n_k)$ be any increasing sequence which is lacunary \`a la Hadamard, i.e. $\underline{\lim} \frac{n_{k+1}}{n_k} > 1$. Then there exists $K  > 0$, such that for any finite sequence $(x_k)$ in $S_1$, we have \begin{eqnarray}\label{nonkhintchine}\frac{1}{K} |||(x_k) ||| \le \n \sum_k x_k e^{i n_k t} \n_{L_1(\T; S_1)} \le K ||| (x_k) |||.\end{eqnarray}
\end{thm}

\begin{rem}\label{uncond}
Under the same condition as in the above theorem, by the equivalence \eqref{nonkhintchine}, it is easy to see that if $(a_k)$ is a bounded sequence in $\C$, then $$ \n\sum_k a_k x_k e^{i n_k t}  \n_{L_1(\T; S_1)} \lesssim  \n \sum_k x_k e^{i n_k t} \n_{L_1(\T; S_1)}.$$ If $(a_k)$ is moreover uniformly separated from 0, i.e. $\inf_k | a_k | > 0$, then $$ \n\sum_k a_k x_k e^{i n_k t}  \n_{L_1(\T; S_1)} \approx  \n \sum_k x_k e^{i n_k t} \n_{L_1(\T; S_1)}.$$
\end{rem}

\begin{defn}
A smoothness $S \subset \N^d$ is said to have Property (O) if there are $\alpha, \beta \in S$ with $ |\alpha| \not\equiv | \beta | \text{ mod 2}$ and $ c= (c(j))$ with $c(j) > 0$ such that: 
\begin{itemize}
\item[(i)] $\langle \alpha, c \rangle = \langle \beta, c \rangle = 1$ 
\item[(ii)] $\langle \gamma, c \rangle \le 1 \text{ for all $\gamma \in S$.}$
\end{itemize}
\end{defn}

\begin{rem}\label{O'}
Assume that $S$ has property (O) and let $\alpha, \beta \in S$ be the two points in $S$ as in the definition of property (O). Then there exists a sequence $(n_k) \subset \N^d$ such that \begin{eqnarray}\label{nk}\lim_k \inf_j n_k(j) = \infty\end{eqnarray} and \begin{eqnarray}\label{rho} \rho = \min\{ \inf_k \frac{| \sigma_\alpha(n_k) |}{Q_S(n_k)^{1/2}},  \inf_k \frac{| \sigma_\beta(n_k) |}{Q_S(n_k)^{1/2}}\} > 0.\end{eqnarray} For the proof, see Proposition 1.2 in \cite{PW}.
\end{rem}

We end this section by stating the following technical proposition from \cite{PW}.
\begin{prop}(Pe{\l}czy{\'n}ski {\&} Wojciechowski )\label{techprop}
Let $S \subset \N^d$ be a smoothness. Then given $\varepsilon$ with $0 < \varepsilon < 1$ and $ D = 1, 2, \cdots$ there exists $\rho = \rho (D, \varepsilon) > 1$ such that, for every $m, n \in \Z^d$, if $\min_{1 \le j \le d} | n(j)| \ge \rho$ and if $\sum_{j = 1}^d | n(j) - m(j)| \le D$ then $$| 1 - Q_S(n)Q_S(m)^{-1} | < \varepsilon;$$ $$ \sum_{\alpha \in S} \left|\frac{|\sigma_\alpha(m)|}{Q_S(m)^{1/2}}- \frac{|\sigma_\alpha(n)|}{Q_S(n)^{1/2}}\right|^2 < \varepsilon^2;$$ $$ \sum_{\alpha \in S} \left|\frac{\sigma_\alpha(m)}{Q_S(m)^{1/2}}- \frac{\sigma_\alpha(n)}{Q_S(n)^{1/2}}\right|^2 < \varepsilon^2.$$
\end{prop}

\section{Main result}
\begin{thm}\label{mainthm}
If the smoothness $S$ satisfies Property (O), then there exists a completely bounded Paley projection $P_\Lambda: W^S_1(\T^d) \to W^S_1(\T^d)$ associated to some infinite sequence $\Lambda = (n_k) \subset \Z^d$. Moreover, the linear map $\hat{P}: W^S_1(\T^d)_\Lambda \rightarrow C + R$ defined by $$\hat{P} f = \sum_{k= 1}^\infty Q_S(n_k)^{1/2} \hat{f}(n_k) e_k $$ is a complete isomorphism, where $(e_k)_{k =1}^\infty$ is an orthonomal basis of $C+ R$.
\end{thm}

The following lemma will be used in the proof of Theorem \ref{mainthm}.

\begin{lem}\label{essential_lem}
Assume that $\Sigma  \subset \Z^d$ is an infinite subset satisfies the conditions $n(1) \ge 1, \text{ $\forall n \in \Sigma \setminus \{0\}$}$ and the projection to the first coordinate $\Sigma \rightarrow \N$ defined by $n  \mapsto  n(1)$ is injective. Assume moreover that $\Lambda = (n_k)_{k = 1}^\infty$ is an infinite sequence in $\Sigma$ such that  $$\inf_k \frac{n_k(1)}{n_{k-1}(1)}  > 1.$$ Then the natural map $$P_{\Sigma, \Lambda}: L_1(\T^d)_{\Sigma} \rightarrow L_1(\T^d)_\Lambda$$ is completely bounded and $L_1(\T^d)_{\Lambda}$ is completely isomorphic to $C + R$.
\end{lem}

\begin{proof}
We shall prove that the projection $L_1(\T^d; S_1)_\Sigma \rightarrow L_1(\T^d; S_1)_\Lambda$ is bounded. Let $\Gamma$ be the image of the first projection $\Sigma \to \N$. The injectivity of $n \mapsto n(1)$ on $\Sigma$ implies that there is a map $m: \Gamma \to \Z^{d-1}$ such that $n = (n(1), m(n(1)))$ for all $n \in \Sigma$. 
We write $x \in \T^d$ as a pair $x = (t, y) \in \T \times \T^{d-1}$.  To each $y \in \T^{d-1}$ and $g \in L_1(\T^d; S_1)$ we associate with a function $g_y: \T \rightarrow S_1$ defined by $g_y(t) = g(t, y)$. If $g \in L_1(\T^d; S_1)_\Sigma$, then $$ g(t, y) \sim \sum_{n \in \Sigma} \hat{g}(n) e^{i\langle (t, y), n \rangle} = \sum_{n(1) \in \Gamma}  \hat{g}(n) e^{i t n(1)} e^{i \langle y, m(n(1)) \rangle}.$$ This implies that $\text{spec}(g_y) \subset \Gamma \subset \N$ and $$ \hat{g_y}(n(1)) = \hat{g}(n)e^{i \langle y, m(n(1)) \rangle}.$$ By \cite{Lust_Pisier}, as operator space, $L_1(\T)_\Gamma$ is completely isomorphic to $C + R$. Hence for any fixed $y \in \T^{d-1}$, $$ \n \sum_{n(1) \in \Gamma}  \hat{g}(n) e^{i t n(1)} e^{i \langle y, m(n(1)) \rangle}\n_{L_1(\T; S_1)} \approx \n \sum_{n(1) \in \Gamma}  \hat{g}(n) e^{i t n(1)} \n_{L_1(\T; S_1)}.$$ It follows that  \begin{eqnarray*} \n g \n_{L_1(\T^d; S_1)}& =& \int_{\T^{d-1}} \n \sum_{n(1) \in \Gamma}  \hat{g}(n) e^{i t n(1)} e^{i \langle y, m(n(1)) \rangle}\n_{L_1(\T; S_1)} dy \\ & \approx &\n \sum_{n(1) \in \Gamma}  \hat{g}(n) e^{i t n(1)} \n_{L_1(\T; S_1)} \\ &= &\n \sum_{n(1) \in \Gamma}  \hat{g}(n) e^{i t n(1)} \n_{H^1(\T; S_1)}.\end{eqnarray*}
Similarly, we have $$\n P_{\Sigma, \Lambda} g\n_{L_1(\T^d; S_1)} \approx \n \sum_{k = 1}^\infty \hat{g}(n) e^{i t n_k(1)} \n_{H^1(\T; S_1)}.$$ The sequence $(n_k(1))_{k = 1}^\infty$ is lacunary, thus we can apply Corollary 0.4 in \cite{Lust_Pisier} to obtain $$\n \sum_{k = 1}^\infty \hat{g}(n) e^{i t n_k(1)} \n_{H^1(\T; S_1)} \lesssim \n\sum_{n(1) \in \Gamma}  \hat{g}(n) e^{i t n(1)} \n_{H^1(\T; S_1)} .$$ Combining the above inequalities, we have $$ \n P_{\Sigma, \Lambda} g \n_{L_1(\T^d; S_1)} \lesssim \n g \n_{L_1(\T^d; S_1)}.$$ This completes the proof that $P_{\Sigma, \Lambda}: L_1(\T^d)_{\Sigma} \to L_1(\T^d)_\Lambda$ is completely bounded.

The fact that $L_1(\T^d)_\Lambda \approx C+R$ is then easy. Indeed, if $h \in L_1(\T^d; S_1)_\Lambda$, then \begin{eqnarray}\label{CR} \n h \n_{L_1(\T^d; S_1)_\Lambda} \approx \n \sum_{k = 1}^\infty \hat{h}(n) e^{itn_k(1)} \n_{H^1(\T; S_1)} \approx||| (\hat{h}(n))|||.\end{eqnarray} In other words, $L_1(\T^d)_\Lambda \simeq C + R$ completely isomorphically.
\end{proof}

\begin{rem}\label{usefulrem}
In the situation of Lemma \ref{essential_lem}, the map $$\widetilde{P}_{\Sigma, \Lambda}: L_1(\T^d)_{\Sigma} \rightarrow C+R$$ defined by $\widetilde{P}_{\Sigma, \Lambda} f = \sum_{k=1}^\infty \hat{f}(n_k) e_k$, where $e_k$ is an orthonormal basis of $C+R$, is completely bounded.
\end{rem}

\begin{proof}[Proof of Theorem \ref{mainthm}]
Our proof follows the proof of Proposition 2.2 in \cite{PW}. Let $\alpha, \beta \in S$ and $(n_k) \in \N^d$ be as in Remark \ref{O'}. Since $| \alpha | \not \equiv | \beta| \text{ mod 2,}$ one can assume that for all $k$, $$\text{sign}(\frac{\sigma_\alpha(n_k)}{\sigma_\beta(n_k)}) = i^{| \alpha| - | \beta |} = \tau,$$ $$\text{sign}(\frac{\sigma_\alpha(-n_k)}{\sigma_\beta(-n_k)}) = (-i)^{| \alpha| - | \beta |} = - \tau.$$ Here $\text{sign}(z) : = \frac{z}{| z |}$ for $z \in \C\setminus 0$. 

Replacing, if necessary, the sequence $(n_k)$ by a rapidly increasing subsequence, we can assume without loss of generality that the sequence $(n_k)$ satisfies the conditions: \begin{itemize}\item[(i)] $\sum_{r=1}^{k-1} \sum_{j=1}^d n_r(j) < \min_j n_k(j)  \text{ for $k = 2, 3, \cdots,$}$ \item[(ii)] $ \lim_k \frac{|\sigma_\alpha(-n_k)|}{|\sigma_\beta(-n_k)|} = \lim_k \frac{|\sigma_\alpha(n_k)|}{|\sigma_\beta(n_k)|}  = \ell > 0,$ \item[(iii)] $ \sum_{k = 1}^\infty | \sigma_\alpha(-n_k) +\tau \ell \sigma_\beta(-n_k)| Q_S(n_k)^{-1/2} \\ =\sum_{k = 1}^\infty \Big| |\sigma_\alpha(n_k) |- \ell | \sigma_\beta(n_k)|\Big| Q_S(n_k)^{-1/2} < \frac{1}{2},$ \item[(iv)] $\sum_{k=1}^\infty \sum_{m \in B_k} | \sigma_\alpha(-m)  + \tau \ell \sigma_\beta(-m) | Q_S(-m)^{1/2} < 1$,\end{itemize} where $B_1 = \{n_1\}$ and for $k =2, 3, \cdots, $ $$ B_k = \Big\{m \in \Z^d: \sum_{j= 1}^d| m(j) - n_k(j) | \le \sum_{r = 1}^{k-1} \sum_{j=1}^d  n_r(j) \Big\}.$$ Notice that item (iv) follows from (iii) , Proposition \ref{techprop} and also the assumption that $(n_k)$ increase sufficiently fast.

Define $M: W^S_1(\T^d) \rightarrow L_1(\T^d)$ by $$Mf = \partial^\alpha f  + \tau \ell \partial^\beta f - \sum_{k = 1}^\infty \sum_{m \in B_k} ( \sigma_\alpha(-m) + \tau \ell \sigma_\beta (-m) )\hat{f}(-m) e^{-i\langle \cdot, m\rangle}. $$ Then $M$ is completely bounded. Indeed, consider the map $M \otimes I_{S_1}: W^S_1(\T^d; S_1) \rightarrow L_1(\T^d; S_1)$. If $g \in W^S_1(\T^d; S_1)$, then $$\n \partial^\alpha g\n_{L_1(\T^d; S_1)}+ \n \partial^\beta g\n_{L_1(\T^d; S_1)}\le \n g \n_{W^S_1(\T^d; S_1)}.$$ Remember that $\sigma_\gamma (n) \hat{g}(n)  = \int_{\T^d} \partial^\gamma g(x) e^{-i\langle x, n \rangle} dx$, hence for any $\gamma \in S$, $$ \n \sigma_\gamma(n) \hat{g}(n) \n_{S_1} \le \n \partial^\gamma g\n_{L_1(\T^d;S_1)} \le \n g\n_{W^S_1(\T^d; S_1)}. $$ Hence $Q_S(n)^{1/2}\n \hat{g}(n) \n_{S_1} \le | S | \cdot \n g \n_{W^S_1(\T^d; S_1)}.$ Combining with (iv), we have \begin{eqnarray*} & & \n\sum_{k = 1}^\infty \sum_{m \in B_k} ( \sigma_\alpha(-m) + \tau \ell \sigma_\beta (-m) )\hat{g}(-m) e^{-i\langle \cdot, m\rangle} \n_{L_1(\T^d; S_1)} \\ & \le & | S | \cdot \n g \n_{W^S_1(\T^d; S_1)} .\end{eqnarray*} Hence $M \otimes I_{S_1}$ is bounded and $M$ is completely bounded.

Next, consider the measure $\mu_R$ on $\T^d$ given by the Riesz product $$R(x) = \prod_{k=1}^\infty ( 1 + cos\langle x, n_k\rangle) = \prod_{k=1}^\infty ( 1 + \frac{1}{2} e^{ i\langle x, n_k\rangle} +  \frac{1}{2} e^{ -i\langle x, n_k\rangle} ).$$ Then the convolution map $M_R: L_1(\T^d) \rightarrow L_1(\T^d)$ defined by $M_R f = f \ast \mu_R$ is obviously completely contractive. In the definition of this Riesz product, we assume that $\frac{n_{k+1}(1)}{n_k(1)} \ge 3$, for all $k = 1, 2, \cdots$. Notice that we have $$\text{spec}(\mu_R) = \Big\{ d_1 n_1 + d_2 n_2 + \cdots + d_k n_k: k \in \N, d_1, d_2, \cdots d_k \in \{ -1, 0, 1\} \Big\}.$$ 

\textbf{Claim A}: $\text{spec}(\mu_R) \subset \{0\} \cup \bigcup_{k=1}^\infty(B_k \cup (-B_k)).$ Indeed, if $m \in \text{spec}(\mu_R)\setminus \{0\}$, then there exist $k\ge 1$ and $d_1, d_2, \cdots, d_k \in \{ -1, 0, 1\}$, such that $d_k \ne 0$ and $m = d_1 n_1 + d_2 n_2 + \cdots  + d_k m_k$. Replacing $m$ by $-m$, if necessary, one may assume that $d_k = 1$, then $ m - n_k = \sum_{r=1}^{k-1} d_r n_r$, it follows that $\sum_{j=1}^d | m(j) - n_k(j)| \le \sum_{r=1}^{k-1}\sum_{j= 1}^d n_r(j)$, i.e. $m \in B_k$.

\textbf{Claim B}: The projection to the first coordinate $\text{spec}(\mu_R) \rightarrow \Z$ is injective. Indeed, if $n, m \in \text{spec} (\mu_R)$ such that $n(1) = m(1)$, suppose that $n = d_1 n_1 + d_2 n_2 + \cdots + d_k n_k$ and $m = d_1' n_1 + d_2' n_2 + \cdots d_{k'}' n_{k'}$, then by a simple computation (cf. e.g. \cite{Riesz_prod}), we have $k = k'$ and $d_1= d_1'$, $d_2 = d_2', \cdots, d_k = d_k'$, hence $n = m$. In other words, the projection to the first coordinate $\text{spec}(\mu_R) \rightarrow \Z$ is injective. 

Let $\Sigma = \{0\} \cup \bigcup_{k=1}^\infty B_k$. It can be easily checked that the image $\text{Im}(M_R M)$ of the composition operator $M_R M$ is contained in $L_1(\T^d)_\Sigma$. By the definition of $B_k$ and condition (i) on the sequence $(n_k)$ , if $m \in B_k$, then $$m(1) \ge n_k(1)  - \sum_{r=1}^{k-1}\sum_{j=1}^d n_r(j) > 0.$$ 

We are now in the situation of Lemma \ref{essential_lem}, thus we obtain a completely bounded projection $P_{\Sigma, \Lambda}: L_1(\T^d)_\Sigma \rightarrow L_1(\T^d)_\Lambda$. By composition, we obtain the following completely bounded map $$ P_{\Sigma, \Lambda} M_R M : W^S_1(\T^d) \rightarrow L_1(\T^d)_\Lambda.$$ By computation, we have\begin{eqnarray*}P_{\Sigma, \Lambda} f = \sum_{k = 1}^\infty \rho_k Q_S(n_k)^{1/2} \hat{f}(n_k) e^{i \langle \cdot, n_k\rangle},\end{eqnarray*} where $\rho_k = \frac{\sigma_\alpha(n_k) + \tau \ell \sigma_\beta(n_k)}{2Q_S(n_k)^{1/2}}$. By \eqref{rho}, $$| \rho_k | = \frac{1}{2}(| \sigma_\alpha(n_k) | + \ell |\sigma_\beta(n_k)|) Q_S(n_k)^{1/2} \ge \frac{1}{2}\rho (1+ \ell).$$ On the other hand, it is obvious that $|\rho_k| \le \frac{1}{2}(1+\ell)$. Let $g: \T^d \rightarrow S_1$, then \begin{eqnarray*} \n P_\Lambda g \n_{W^S_1(\T^d; S_1)} & = & \sum_{\gamma \in S} \n \sum_{k = 1}^\infty \sigma_\gamma(n_k) \hat{g}(n_k) e^{i\langle \cdot, n_k\rangle} \n_{L_1(\T^d;S_1)} \\ \text{ By \eqref{CR}} &\approx& \sum_{\gamma \in S} \n \sum_{k = 1}^\infty \sigma_\gamma(n_k) \hat{g}(n_k) e^{i t n_k(1)} \n_{L_1(\T;S_1)} \\ \text{By Remark \ref{nonkhintchine}}& \lesssim &  \sum_{\gamma \in S} \n \sum_{k = 1}^\infty Q_S(n_k)^{1/2} \hat{g}(n_k) e^{i t n_k(1)} \n_{L_1(\T;S_1)} \\ \text{By Remark \ref{nonkhintchine} and $| S | < \infty$} & \approx & \n \sum_{k = 1}^\infty \rho_k Q_S(n_k)^{1/2} \hat{g}(n_k) e^{i t n_k(1)} \n_{L_1(\T;S_1)} \\ \text{By \eqref{CR}}&\approx & \n \sum_{k = 1}^\infty \rho_k Q_S(n_k)^{1/2} \hat{g}(n_k) e^{i \langle \cdot, n_k\rangle}\n_{L_1(\T^d;S_1)} \\ & = & \n (P_{\Sigma, \Lambda} M_R M \otimes I_{S_1}) g \n_{L_1(\T^d; S_1)} \\ & \lesssim& \n g \n_{W^S_1(\T^d; S_1)}.\end{eqnarray*} This completes the proof that $P_\Lambda: W^S_1(\T^d) \rightarrow W^S_1(\T^d)$ is completely bounded. For the second assertion of the theorem, we only need to notice that by Remark \ref{O'}, $| \sigma_\alpha(n_k) | \ge \rho Q_S(n_k)^{1/2}$ for all $k$, hence if $g \in W^S_1(\T^d;S_1)_\Lambda$, then \begin{eqnarray*} \n g \n_{W^S_1(\T^d; S_1)}& \ge& \n \partial^\alpha g \n_{L_1(\T^d; S_1)} = \n \sum_{k=1}^\infty \sigma_\alpha(n_k) \hat{g}(n_k) e^{i \langle \cdot, n_k\rangle} \n_{L_1(\T^d; S_1)} \\ & \gtrsim& \n \sum_{k=1}^\infty Q_S(n_k)^{1/2} \hat{g}(n_k) e^{i \langle \cdot, n_k\rangle} \n_{L_1(\T^d; S_1)} \\ &\approx &\n \sum_{k=1}^\infty Q_S(n_k)^{1/2} \hat{g}(n_k)\otimes e_k \n_{S_1[C+R]}.\end{eqnarray*}
\end{proof}

Using Theorem \ref{mainthm}, then by a classical transference method, we have the following corollary, for the definition of quantum torus  $\T_\theta^d$ and harmonic analysis on it, we refer to the paper \cite{Xu_Yin}.
\begin{cor}
Under the same condition of Theorem \ref{mainthm}, there exists a completely bounded Paley projection $P_\Lambda: W^S_1(\T^d_\theta) \rightarrow W^S_1(\T^d_\theta)$ associated to some infinite sequence $\Lambda = (n_k) \subset \Z^d$. 
\end{cor}

\section*{Acknowledgements}
The author would like to thank Quanhua Xu for inviting him to Universit\'e Franche-Comt\'e and his constant encouragement.

\end{document}